\documentclass[12pt]{amsart}
\usepackage{amsmath,amstext,amsfonts,amsthm,amssymb,hyperref}
\usepackage{eucal}
\usepackage{diagrams}                                      %
\diagramstyle[scriptlabels,PostScript=dvips]               %
\newarrow{Dash}{}{dash}{}{dash}{}                          %
\newarrow{Equal}=====
\swapnumbers
\newtheorem{thm}[subsubsection]{Theorem}
\newtheorem{lem}[subsubsection]{Lemma}
\newtheorem{prp}[subsubsection]{Proposition}
\newtheorem{crl}[subsubsection]{Corollary}

\newtheorem*{Lem}{Lemma}
\newtheorem*{Prp}{Proposition}
\newtheorem*{Crl}{Corollary}

{  \theoremstyle{definition}
           \newtheorem{dfn}[subsubsection]{Definition}
           \newtheorem{exm}[subsubsection]{Example}
           \newtheorem{rem}[subsubsection]{Remark}

           \newtheorem*{Exm}{Example}

}

\newcommand{\Cat}{\mathtt{Cat}}

\newcommand{\Coc}{\mathtt{Coc}}

\newcommand{\colim}{\operatorname{colim}}

\newcommand{\CSS}{\mathrm{CSS}}

\newcommand{\Fin}{\mathit{Fin}}
\newcommand{\Fun}{\operatorname{Fun}}

\newcommand{\Hom}{\mathrm{Hom}} 
\newcommand{\Ho}{\operatorname{Ho}}
\newcommand{\hor}{\mathrm{hor}}

\newcommand{\lax}{\mathrm{lax}}
\newcommand{\Left}{\mathbf{L}}
\newcommand{\lrarrows}{\pile{\rTo\\ \lTo}}

\newcommand{\Map}{\operatorname{Map}}
\newcommand{\Mor}{\mathit{Mor}}

\newcommand{\Ob}{\operatorname{Ob}}

\newcommand{\RelCat}{\mathtt{RelCat}}
\newcommand{\rlarrows}{\stackrel{\rTo}{\lTo}}
\newcommand{\Right}{\mathbf{R}}

\newcommand{\SM}{\mathrm{SM}}
\newcommand{\Sp}{\cS}

\newcommand{\sCat}{\mathtt{sCat}}
\newcommand{\ssSet}{\mathtt{ssSet}}
\newcommand{\sSet}{\mathtt{sSet}}
\newcommand{\Sing}{\operatorname{Sing}}

\newcommand{\ver}{\mathrm{ver}}
\newcommand{\wt}{\widetilde}

\newcommand{\Alg}{\mathtt{Alg}}

\newcommand{\bF}{\mathbb{F}}

\newcommand{\bW}{\mathbb{W}}

\newcommand{\cC}{\mathcal{C}}
\newcommand{\cD}{\mathcal{D}}
\newcommand{\cE}{\mathcal{E}}

\newcommand{\cK}{\mathcal{K}}
\newcommand{\cL}{\mathcal{L}}
\newcommand{\cM}{\mathcal{M}}

\newcommand{\cO}{\mathcal{O}}

\newcommand{\cS}{\mathcal{S}}

\newcommand{\cW}{\mathcal{W}}

\newcommand{\fC}{\mathfrak{C}}

\newcommand{\Z}{\mathbb{Z}}
\newcommand{\Q}{\mathbb{Q}}

\begin{document}

\title[]{Dwyer-Kan localization revisited}
\author{Vladimir Hinich}
\address{Department of Mathematics, University of Haifa,
Mount Carmel, Haifa 3498838,  Israel}
\email{hinich@math.haifa.ac.il}
\begin{abstract}
A version of Dwyer-Kan localization in the context of $\infty$-categories and simplicial categories is presented. 
Some results of the classical papers \cite{DK1,DK2,DK3} are reproven and generalized.
We prove that a Quillen pair of model categories gives rise to an
adjoint pair of their DK localizations (considered as $\infty$-categories). 
We study families of $\infty$-categories and present a result on localization 
of a family of $\infty$-categories. This is applied to localization of symmetric monoidal $\infty$-categories where we were able to get only partial results.  
\end{abstract}
\dedicatory{To the memory of Daniel Kan}
\maketitle
\section*{Introduction}

This paper was devised as an appendix to \cite{H.R} intended to describe necessary prerequisites about localization in $(\infty,1)$-categories. The task turned out to be more serious and
more interesting than was originally believed. This is why we finally decided to 
present it as a separate text.

The paper consists of three sections. In Section ~\ref{app:nerve} we present a version 
of Dwyer-Kan localization in the context of $\infty$-categories \footnote{in the sense of Lurie~\cite{L.T}} 
and simplicial categories. The original approach of Dwyer and Kan \cite{DK1,DK2,DK3} is replaced, in the context of $\infty$-categories, with a description using universal property. We compare the approaches showing that the homotopy coherent nerve carries hammock localization 
of fibrant simplicial categories to a localization of $\infty$-category in our universal sense.  

A very important example of Dwyer-Kan localization is the underlying
$\infty$-category of a model category. We reprove the classical result \cite{DK3}, Proposition 5.2, and prove a generalization of \cite{DK3}, 4.8, giving various equivalent descriptions of this localization.
Our approach is based on Key Lemma~\ref{sss:keylemma} which
gives a convenient criterion for a functor $f:\wt C\to C$ between (conventional) categories to be a DK localization.
 
Applying Key Lemma to the  case $\wt C$ is a category of resolutions
of objects in $C$, we are able to easily deduce most of the results 
about equivalence of different descriptions of the underlying $\infty$-category
 of a model category.
 
Another result of Section ~\ref{app:nerve} is Proposition \ref{prp:nerve-adj}
saying that a Quillen pair of model categories gives rise
to an adjoint pair of their underlying $\infty$-categories. This was previously proven 
for a simplicial Quillen adjunction, see \cite{L.T}, 5.2.4. 
  
In Section~\ref{ss:quasi} of the paper we present a way to simultaneously localize
a family of $\infty$-categories. Under some conditions described
in \ref{dfn:mccf}, localization of a fiber of $f:C\to D$ is equivalent
to the homotopy fiber of the map of localizations. 

This result is applicable when one studies the $\infty$-category of pairs $(A,M)$
where $A$ is a dg algebra and $M$ is $A$-module, as (co)fibered over the $\infty$-category of
dg algebras. This is how we use it in~\cite{H.R}.
 
In the last Section~\ref{sec:SM} we make an attempt to understand the universal meaning of
SM $\infty$-category underlying a SM model category. Let $\cC$ be a symmetric monoidal model category. The homotopy category $\Ho(\cC)$ has a symmetric monoidal structure with the tensor 
product defined as the left derived fucnctor of the tensor product in $\cC$. The canonical localization functor $\cC\rTo\Ho(\cC)$ is lax symmetric monoidal. It is not difficult to
produce a SM $\infty$-category whose homotopy category is equivalent to $\Ho(\cC)$:
one defines it as a DK localization of the full subcategory $\cC^c\subset\cC$ spanned by the cofibrant objects of $\cC$ (see Lurie \cite{L.HA}, 4.1.3). However, it is not clear
in general how to present the passage from $\cC$ to the DG localization of $\cC^c$
as a universal construction.

We suggest to define a right SM localization (of a SM $\infty$-category $C$ with respect to a
collection $W$ of arrows) as a lax SM functor $C\rTo D$ carrying $W$ to equivalences, universal with respect to this property, and equivalent to the usual DK localization once the SM structure
is forgotten. In a special case $C$ is the category of complexes over a commutative ring
we are able to prove the existence of right localization, see \ref{exm:complexes}.
We do not know general conditions which would ensure its existence.

\subsubsection*{Acknowledgements}
Parts of this paper were written during author's visit to MIT and IHES. I am grateful to these institutions for hospitality and excellent 
working conditions.  I am very grateful to the referee for numerous corrections and suggestions.

\section{$\infty$-Localization. $\infty$-category of a model category}
\label{app:nerve}

In \ref{ss:loc} we present the notion of $\infty$-localization in the context
of $\infty$-categories. We work in the setting of $\infty$-categories as defined and developed in
\cite{L.T} and \cite{L.HA}. Localization of an $\infty$-category along a 
collection of arrows is defined by a universal property; it can 
be easily expressed in terms of fibrant replacement in the model category of marked
simplicial sets, \cite{L.T}, Chapter 3.

A more explicit construction of $\infty$-localization can be given in terms of Dwyer-Kan localization of simplicial categories. The equivalence of two approaches is "almost obvious". This is why we prefer to extend the name "Dwyer-Kan localization" to include the $\infty$-localization of $\infty$-categories.
\footnote{Another reason is the wish to avoid confusion with much more narrow Lurie's notion of localization, see~\cite{L.T}, 5.2.7, which rather deserves the name {\sl Bousfield localization}.} 

We use the notion of $\infty$-localization to define the underlying 
$\infty$-category of an arbitrary model category. This notion generalizes the notion of an underlying $\infty$-category of a simplicial model category as defined in \cite{L.T}, A.2.

In Section \ref{ss:weaks} we study weak simplicial model categories.
These are model categories with a structure of a simplicial category
which is compatible in a weak sense with the model structure, see Definition \ref{dfn:weak}.
Such sort of compatibility has, for instance, the category of 
complexes, or the category of commutative DG algebras over a field of characteristic zero.

Our Proposition~\ref{sss:prp} extends to weak simplicial model categories
Theorem 4.8 from \cite{DK3} saying, in particular, that the underlying
$\infty$-category in this case is equivalent to the nerve of
the simplicial category of fibrant cofibrant objects.

In \ref{prp:nerve-adj} we show that a Quillen pair of model categories gives rise to an adjoint pair of functors between the respective underlying $\infty$-categories. The result was previously
known for a simplicial Quillen adjunction (see \cite{L.T}, 5.2.4) and, in the language of simplicial categories, for a Quillen equivalence, see \cite{DK2}.

\subsection{Dwyer-Kan localization in $\infty$-categories}
\label{ss:loc}
\subsubsection{Total localization}
\label{sss:tloc}

The $\infty$-category of spaces $\Sp$ is the full subcategory of $\Cat_\infty$
spanned by $\infty$-categories whose all arrows are equivalences.  The tautological embedding
$$ i:\Sp\rTo\Cat_\infty$$
has both left and right adjoints  which we will denote $L$ and $K$ respectively.

The existence of adjoints can be shown as follows. We can realize $\Cat_\infty$ as the $\infty$-category underlying the simplicial 
model category $\sSet^+$ of marked simplicial sets, and $\Sp$ as underlying the category $\sSet^+$ endowed with a localized model structure, see \cite{L.T}, 3.1.5.6.

Thus, the fully faithful embedding $i:\Sp\to\Cat_\infty$ admits a left adjoint
$L:\Cat_\infty\to\Sp$ which defines a localization in the sense of Lurie, \cite{L.T}, 5.2.7.2. 

The right adjoint functor $K$ assigns to an $\infty$-category $X$ the maximal Kan subcomplex $K(X)$.

This formally implies that the composition $\cL=i\circ L$ is left adjoint to
the composition $\cK=i\circ K$. The unit of adjunction defines a 
canonical map $X\to\cL(X)$. The functors $L$ and $\cL$ are 
{\sl total $\infty$-localization functors}. If $C$ is an $\infty$-category, 
$\cL(C)$ is presented by a Kan fibrant replacement of $C$.

\subsubsection{Marked $\infty$-categories and their $\infty$-localization}

A marked $\infty$-category is, by definition, a pair $(C,W)$ with $C$ an 
$\infty$-category and $W$ a collection of arrows in $C$. A marking $W$ 
is {\sl saturated} if there exists a map $C\to D$ of $\infty$-categories such that $W$ is the preimage of the collection of equivalences in $D$. Since equivalences in $D$ are precisely the arrows whose image in the homotopy category $\Ho(D)$ is an isomorphism, a saturated marking of $C$ is always defined by a subcategory $W\subset C$ in the sense of \cite{L.T}, 1.2.11.
In what follows all markings will be 
assumed saturated. Marked $\infty$-categories form an $\infty$-category
$\Cat^+_\infty$ which is the full subcategory of $\Fun(\Delta^1,\Cat_\infty)$
spanned by arrows $W\to C$ determined by saturated markings $W$ of $C$.

Given a map $f:W\to C$ in $\Cat_\infty$, one defines $\cL(f)$ or $\cL(C,W)$ as the object (co)representing the $\infty$-functor
\begin{equation}\label{eq:universal-DK}
\Map(\cL(f),X)=\Map(C,X)\times_{\Map(W,X)}\Map(W,\cK(X))
\end{equation}
from $\Cat_\infty$ to $\Sp$.
The definition immediately implies the formula
\footnote{Recall that the colimit is meant to be in $\Cat_\infty$, that is, in the
 "infinity sense".}
$$ \cL(f)=\cL(W)\coprod^WC.$$

\subsubsection{Description in terms of  marked model structure}

Let $C$ be an $\infty$-category. We will check that the total localization
$\cL(C)$ is represented by a fibrant replacement $\tilde{C}$ of the 
marked simplicial set $C^\sharp=(C,C_1)$.
In fact, let $X$ be an $\infty$-category. We have a commutative diagram
\begin{equation}\label{}
\begin{diagram}
\Map(\tilde C,\cK(X)) & \rTo^f &\Map(C,\cK(X)) \\
\dTo^g & & \dTo \\
\Map(\tilde C,X) & \rTo &\Map(C,X)
\end{diagram},
\end{equation}
where all $\Map$ spaces are taken in $\Cat_\infty$.
Note that $\cK(X)$ is Kan.
Therefore, the source and the target of $f$ can be calculated in $\sSet^+$;
therefore, $f$ is a weak equivalence. On the other hand, $\tilde C$ is a
fibrant replacement of $C^\sharp$, so is also Kan. Therefore, $g$ is a bijection.
This proves the assertion.

The same is true for a general localization. Let $f:W\to C$ be as above.
Choose fibrant replacements $W^\sharp\to\tilde W$ and 
$\tilde W\sqcup^{W^\sharp}(C,W)\rTo\tilde C$ in the category of marked simplicial sets. 
Since the marked model structure is left proper, the composition $(C,W)\to\tilde C$ is a weak 
equivalence, so that the fibrant replacement $\tilde C$ of $(C,W)$ represents
the localization $\cL(C,W)$.

Thus, we have  
\begin{Prp}
For $(C,W)\in\Cat^+_\infty$ the $\infty$-localization $\cL(C,W)$ is represented by a
fibrant replacement of $(C,W)$ considered as marked simplicial set.
\end{Prp}

Note that one has a tautological map $\tau:\phi\rTo\cL$ of functors 
$\Cat^+_\infty\rTo\Cat_\infty$ where $\phi$ is the functor forgetting the marking of a 
marked $\infty$-category.

\subsection{Dwyer-Kan localization in simplicial categories}

Using the model category structure on simplicial categories (Bergner model structure), Dwyer-Kan localization can be described as the derived functor of a conventional localization.

Given a map $\cW\to\cC$ of simplicial categories, its DK localization can be described
as represented by a conventional localization $\wt\cC[\wt\cW^{-1}]$ where in the diagram
\begin{equation}
\begin{diagram}
\wt\cW & \rTo^p & \cW \\
\dTo^{\wt i} & & \dTo^i \\
\wt\cC & \rTo^q & \cC
\end{diagram}
\end{equation}
$p$ and $q$ are cofibrant replacements and $\wt i$ is a cofibration.

The above definition was suggested by Dwyer and Kan in \cite{DK1}, with an explicit choice 
of cofibrant replacements. In the second paper of the series, \cite{DK2}, another
important variant of the definition, {\sl hammock localization}, weakly equivalent to the above one, was given.  It is worth mentioning   that the hammock localization
$L^H(\cC,\cW)$ admits a localization map $\cC\to L^H(\cC,\cW)$ (the originally defined localization admitted instead a map from a cofibrant replacement $\wt\cC$), and that the simplicial sets $\Map_{L^H(\cC,\cW)}(x,y)$ have an explicit description in terms of diagrams.

Since simplicial categories provide a legitimate model for $(\infty,1)$-categories, it is natural to compare two kinds of localizations.  Let us show in what sense DK localization
and the $\infty$-localization defined in \ref{ss:loc} "are actually the same".

Recall \cite{L.T}, 2.2.5.1,  that there is a Quillen equivalence
\begin{equation}\label{eq:QE}
\fC:\sSet\lrarrows\sCat:N
\end{equation}
between the category of simplicial sets with Joyal model structure and the category of simplicial
categories with Bergner model structure. The right Quillen functor here is {\sl the homotopy coherent nerve} functor which we will simply call the nerve. We will denote by $\Right N$
its derived functor which is calculated as the nerve functor applied to a fibrant replacement.

Let $C$ be a simplicial category. Its total localization is
a map $C\to\tilde C$ such that the map of their derived nerves
 $\Right N(C)\to\Right N(\tilde C)$
is a total localization of $\infty$-categories in the sense of \ref{sss:tloc}. 

By \cite{DK1}, 9.2,
the total DK localization $C\to L(C,C)$ satisfies the above property.

Dwyer-Kan localization represents the $\infty$-localization also in general.
To show this, let $f:W\rTo C$ be a map of $\infty$-categories defined by a saturated marking 
on $C$. Let $\cL(f)=\cL(W)\coprod^WC$ be the localization. Applying the functor $\fC$
to the whole picture, we get a cocartesian diagram
\begin{equation}
\begin{diagram}
\fC(W) & \rTo & \fC(\cL(W)) \\
\dTo & & \dTo\\
\fC(C) & \rTo & \fC(\cL(f))
\end{diagram}
\end{equation}
where $\fC(W)\rTo\fC(C)$ is a cofibration of cofibrant simplicial categories. We already know that the map $\fC(W)\rTo\fC(\cL(W))$ is a total localization, so $\fC(\cL(f))$ is a Dwyer-Kan
localization of $\fC(C)$ with respect to $\fC(W)$.

The following reformulation of what we have just checked is useful.

\begin{prp}\label{prp:2loc}
Let $\cC$ be a fibrant simplicial category and $\cW$ a fibrant simplicial
subcategory of $\cC$ with $\Ob(\cW)=\Ob(\cC)$.
The map $\cC\to L^H(\cC,\cW)$ induces a map of marked simplicial sets
$$(N(\cC),N(\cW))\rTo \Right N(L^H(\cC,\cW))^\natural$$
which is a weak equivalence. 
\end{prp}

\subsection{The $\infty$-category underlying a model category}

Dwyer and Kan suggested their localization as a way to retain the important higher homotopy information in the homotopy category. 

Localization of a model category remains the most important application 
of the theory.
 
Recall that if $\cC_*$ is a simplicial model category and $\cC^{cf}_*$ is the full simplicial subcategory consisting of fibrant cofibrant objects, the nerve $N(\cC^{cf}_*)$
is, according to Lurie, the $\infty$-category underlying the model category $\cC$. Since $\cC^{cf}_*$
represents for simplicial model categories the DK localization, see \cite{DK3}, 4.8,
the following definition seems appropriate.

\begin{dfn}\label{dfn:nerve}
Let $\cC$ be a model category and $\cW$ the full subcategory of weak equivalences.
The $\infty$-category $N(\cC)$ underlying the model category $\cC$ 
(or {\sl the nerve of the model category}) is defined as $\Right N(L^H(\cC,\cW))$.
\end{dfn}
 
Proposition~\ref{prp:2loc} implies that the nerve of a model category $\cC$ can be equivalently defined as a fibrant replacement of the marked simplicial set $(\cC,\cW)$.

\subsubsection{Properties of a nerve}

First of all, note that the map spaces $\Map_{L^H(\cC,\cW)}(x,y)$
(and, therefore, the map spaces of the nerve) have \lq\lq{}the correct homotopy type\rq\rq{} as claims the following theorem.

\begin{thm}\label{thm:maps}(\cite{DK3}, 4.4)
For any cosimplicial resolution $x^\bullet$ of $x$ and simplicial
resolution $y_\bullet$ of $y$ the diagonal of the bisimplicial set
$\Hom_\cC(x^\bullet,y_\bullet)$ is homotopy equivalent to $\Map_{L^H(\cC,\cW)}(x,y)$. Moreover, if $x$ is cofibrant, the same
homotopy type has the simplicial set $\Hom_\cC(x,y_\bullet)$. Similarly, if $y$ is fibrant, the same homotopy type has $\Hom_\cC(x^\bullet,y)$. 
\end{thm}

A very important property of the homotopy category $\Ho(\cC)$ of a model category $\cC$ says that it can be described in different ways:
 as the localization of $\cC$ with respect to weak equivalences in $\cC$;
as the localization of the full subcategory $\cC^c$ 
(resp., $\cC^f$ or $\cC^{cf}$) spanned by cofibrant (resp., fibrant or fibrant cofibrant) objects of $\cC$,  with respect to weak equivalences in this subcategory.
 
Equally important is the existence of different presentations of the
$\infty$-category underlying a model category. In this paper existence of different
presentations of the underlying $\infty$-category is indispensable in proving Proposition~\ref{prp:nerve-adj} below which asserts that a Quillen pair of model categories gives rise to an adjoint pair of the respective underlined categories.

Some of such presentations are given in \cite{DK3}, 5.2 and 4.8.
Here they are.
\begin{prp}\label{prp:DK5.2}
(\cite{DK3}, 5.2)
Let $\cC$ be a model category, $\cC^c$ (resp., $\cC^f$ or $\cC^{cf}$) the full subcategory spanned by the cofibrant (resp., fibrant or fibrant cofibrant) objects. Then the following canonical morphisms
of hammock localizations (with respect to weak equivalences)
are equivalences of simplicial categories.
\begin{equation}
L^H(\cC^f) \rTo L^H(\cC) \lTo L^H(\cC^c).
\end{equation}
\end{prp}

\begin{prp}\label{prp:DK4.8}
(\cite{DK3}, 4.8)
Let $\cC_*$ be a simplicial model category. Then the canonical
morphisms of the following simplicial categories are equivalences.
\begin{equation}
\cC^{cf}_*\rTo L^H(\cC^{cf}_*)\rTo L^H(\cC_*)\lTo L^H(\cC)
\end{equation}
\end{prp}

We would like to have an analog of Proposition~\ref{prp:DK4.8}  for model categories with simplicial structure, more general that simplicial model categories.

We think we have found an easy way of proving all equivalences of this sort. It is based on Key lemma presented below.
 
The lemma is formulated in the language of
$\infty$-localization as presented in \ref{ss:loc}. The proof uses presentation of $(\infty,1)$-categories with complete Segal spaces. \footnote{We are grateful to the referee who found an error in the original proof, and suggested an idea of the present proof.}
 
\subsubsection{Key lemma}
\label{sss:keylemma}

Let $\cC,\cD$ be categories, $f:\cC\rTo\cD$ be a functor. For $x\in\cD$ we denote as
$\cC_x$ the fiber 
$\{(c,\theta)|c\in\cC,\theta:f(c)\stackrel{\sim}{\to}x\}$. 

More generally, for $n$-simplex $\sigma\in N_n(\cD)$ we denote as $\cC_\sigma$
the fiber of the functor $f^{[n]}:\cC^{[n]}\rTo\cD^{[n]}$ at $\sigma$.

Here and below we denote $\cC^{[n]}$ the category of functors $[n]\to\cC$
where $[n]$ is the category consisting of $n$ consecutive arrows.
 
\begin{Lem}Let $f:\cC\to\cD$ be a functor. Assume that
for any $\sigma\in N(\cD)$ the fiber $\cC_\sigma$ 
has a weakly contractible nerve.
Then the functor $f$ presents $\cD$ as an $\infty$-localization of $\cC$ with respect to $W=\{a| f(a)\textrm{ is an isomorphism} \}$.
\end{Lem}
\begin{proof}
 
We start with a simple observation.
Let $\cC$ and $\cD$ be categories. Assume that a functor $f:\cC\rTo\cD$ 
has fibers $\cC_d,\ d\in\cD$, whose nerves $N(\cC_d$) are weakly contractible.

Put $W=\{\alpha\in\Mor(\cC)|f(\alpha)\textrm{ is invertible}\}$. This is a subcategory of $\cC$. We claim that the induced map of the nerves
$$ N(W)\rTo N(K(\cD)),$$
where, as usual, $K(\cD)$ is the maximal subgroupoid of $\cD$, is a weak equivalence.
In fact, we can replace $\cD$ with $K(\cD)$ and $\cC$ with the respective preimage
$W$ since this does not alter the fibers. In this way the claim can be immediately reduced to the case $\cD=BG$, the groupoid with one object and automorphism group $G$.
The contractible fiber of this map is, by definition, the base change of $\cC\to BG$
with respect to the universal covering $EG\to BG$. Thus, $N(\cC)$ has a contractible Galois covering with group $G$, so the map $N(\cC)\to N(BG)$ is a weak equivalence.

\

Going back to our lemma, we apply the above observation to the functors  
$f^{[n]}:\cC^{[n]}\to\cD^{[n]}$. They have weakly contractible fibers, so
for each $n$ one has a weak equivalence 
\begin{equation}
N(\cC,W)_n\rTo N(K(\cD^{[n]}))
\end{equation}
from the $n$-th space of the {\sl classification diagram} $N(\cC,W)$ defined as in \cite{R}, 3.3, 
to the nerve of the maximal subgroupoid of $\cD^{[n]}$.

These equivalences define a Reedy equivalence  $N(\cC,W)\rTo N(\cD,K(\cD))$,
the latter being the {\sl classifying diagram} of $\cD$ in the language of \cite{R}, 3.5. 
The following argument due to C.~Schommer-Pries is borrowed from Mathoverflow discussion~\cite{over}. 

Look  at the diagram comparing three models
for $\infty$-categories, that of relative categories (Barwick-Kan,
\cite{BK1,BK2,BK3}), simplicial categories with the Bergner model structure, and simplicial spaces with complete Segal space structure.
\begin{equation}
\begin{diagram}
\RelCat& & \rTo& & \ssSet \\
&\rdTo^{L^H} & & \ruTo^{N^\CSS} & \\
&& \sCat &&
\end{diagram}
\end{equation}
Here the horizontal arrow is the classification diagram functor
$(C,W)\mapsto  N(C,W)$, $L^H$ is the hammock localization and $N^\CSS$
is the composition of homotopy coherent nerve functor $\sCat\to\sSet$
and the functor $i^!:\sSet\to\ssSet$ defined as in 
Joyal-Tierney \cite{JT}.
Barwick and Kan prove that  the hammock localization
$L^H:\RelCat\to\sCat$ and the Rezk classification diagram functor
$(C,W)\mapsto N(C,W)$ induce an equivalence between the respective
$\infty$-categories.  

  We have to verify the
functor $N^\CSS\circ L^H$ induces a functor isomorphic to Rezk's classification diagram
on the $\infty$-categorical level. Since all the functors involved are equivalences, 
this follows from the uniqueness result of Toen~\cite{T}: the only nontrivial automorphism of $\Cat_\infty$ is the passage to the opposite. So, to prove that our two functors are isomorphic, it is enough to verify that they both preserve the initial vertex
of the category $[1]$. 
\end{proof}

\subsubsection{Proof of \ref{prp:DK5.2}}

Here is the proof of \ref{prp:DK5.2} based on the Key Lemma.

Denote $\wt{\cC}$ the category whose objects
are $\wt X\rTo^p X$ where $\wt X$ is cofibrant and $p$ is a weak equivalence. The functor $f:\wt\cC\rTo\cC$ carries $p:\wt X\to X$ to $X$.
We will check that the requirements of the Key lemma are met, so  the functor $f$ is an $\infty$-localization. This immediately implies that $f$ induces an equivalence of  DK localizations
$ L^H(\wt\cC,\wt\cW)\rTo L^H(\cC,\cW)$.

On the other hand, the functor $g:\wt\cC\rTo\cC^c$ carrying $\wt X\to X$ to $\wt X$,
has a left adjoint, so that the unit and the counit are in $\cW$. Thus, this functor
induces an equivalence of the hammock localizations. Finally, there is a morphism of functors 
$i\circ g\rTo f$, where $i:\cC^c\to\cC$, which belongs to $\cW$, so  $i$ should also induce
an equivalence of the hammock localizations.

In order to check the requirements of the Key lemma, 
we will use the recipe presented in \cite{DSA}, A.3. First of all, we check that the
categories in question have a simply connected nerve; then, using Proposition A.3.3 of \cite{DSA}, prove that the reduced homology of their nerves vanish.

Let $\sigma=(X_0\to\ldots\to X_n)$.

\

{\sl $N(\wt\cC_\sigma)$ is connected.} 

It is convenient to use the model structure on $\cC^{[n]}$ with componentwise
weak equivalences and cofibrations, and with the fibrations defined by the right lifting property with respect to trivial cofibrations.

 An object 
\begin{equation}
\begin{diagram}
P_0 & \rTo & \ldots & \rTo & P_n \\
\dTo^{p_0} &&&& \dTo^{p_n} \\
X_0 & \rTo & \ldots &\rTo & X_n
\end{diagram}
\end{equation}
of $\wt\cC_\sigma$ is called {\sl special} the map $p:P\to X$ is a trivial fibration
in the model category structure on $\cC^{[n]}$ described above.

In more detail, this means that $p_n$ is trivial fibration and 
the commutative squares
\begin{equation}
\begin{diagram}
P_{i-1} & \rTo &    P_i \\
\dTo^{p_{i-1}} && \dTo^{p_i} \\
X_{i-1} &\rTo & X_i
\end{diagram}
\end{equation}
induce a fibration $P_{i-1}\rTo X_{i-1}\times_{X_i}P_i$.

Now for any pair $P,Q\in\wt\cC_\sigma$ with $Q$ special the set $\Hom_{\wt\cC_\sigma}(P,Q)$ is nonempty. This proves connectedness of the nerve of $\wt\cC_\sigma$.
 
\

{\sl $N(\wt\cC_\sigma)$ is simply-connected.} The Poincar\'e groupoid of 
$N(\wt\cC_\sigma)$ is the nerve of the full localization of $\wt\cC_\sigma$ 
(here we mean the "conventional" localization in $\Cat$).

For any $p:P\to X$ in $\wt\cC_\sigma$ we construct a cylinder object
\begin{equation}\label{eq:cyl}
P\sqcup P  \rTo^{i_0\sqcup i_1}  \wt P\rTo^q P
\end{equation}
 so that $i_0\sqcup i_1$ is a cofibration and $q$ is a trivial fibration in $\cC^{[n]}$.

Now any pair of arrows $a_0,a_1:P\rTo Q$ with special $Q$ can be extended to a map 
$a:\wt P\rTo Q$ so that $a_j=a\circ i_j$. Finally, given a closed path
\begin{equation}\label{eq:path}
P^0\to P^1\leftarrow \ldots P^n\leftarrow P^0
\end{equation}
in $\wt\cC_\sigma$, choose a special $Q$ and an arrow $P^k\to Q$ for each $k$.
Since all triangles with vertices $P^k,P^{k+1}$ and $Q$ become commutative
in the localization, the image of the path~(\ref{eq:path}) in the localization
is trivial. This proves simply-connectedness of $N(\wt\cC_\sigma)$.

{\sl $N(\wt\cC_\sigma)$ has vanishing reduced homology.} Choose a special
$q:Q\to X$ in $\wt\cC_\sigma$. The functor $Q:[n]\to\cC^c$ gives rise to a simplex
in $\cC$ which we denote $\tau$. 
\begin{equation}\label{eq:red}
\wt\cC_{\tau}\rTo \wt\cC_\sigma
\end{equation}
defined by the composition with $q$. The first category has a final 
object, so its nerve is contractible. The fiber of  (\ref{eq:red})
at $P\to X$ is of form $\wt\cC_{\sigma'}$ with the simplex $\sigma'$
defined by the componentwise fiber product $Q\times_XP$. 
Lemma A.3.3 of \cite{DSA} claims in this case (by induction) that
the reduced homology of $N(\wt\cC_\sigma)$ vanishes.

Proposition~\ref{prp:DK5.2} is proven.

\

The following result (proven in \cite{DK2} for model categories with functorial decomposition) is deduced by precisely the same reasoning.

\begin{prp}\label{prp:extra}
Let $\cC$ be a model category. Then the embedding $\cC^{cf}\to\cC^f$
induces an equivalence of hammock localizations.
\end{prp}
\qed

\

\subsection{Model categories with a simplicial structure}
\label{ss:weaks}
We will generalize Proposition~\ref{prp:DK4.8} to model categories
having a simplicial structure satisfying some (but not all) properties of a simplicial model category.

A typical example of such simplicial structure on a model category is the one one the category of complexes $C(k)$ or the one
on a category of DG algebras (over any operad) in case the ground ring $k$ contains the rational numbers. The structure presented below is not self-dual. So, formally speaking, there is a dual notion (existence of
weak cylinders instead of weak paths). However, we do not know any meaningful example of such structure, so we will not mention it in the sequel.

\

\subsubsection{Weak path functors}
Let $\cC$ be a simplicial category.
 
We will assume that for any simplicial set $K$ the functor
\begin{equation}\label{eq:weakpath}
Y\mapsto\Hom(K,\Map_\cC(Y,X))
\end{equation}
is representable. The representing object will be denoted $X^K$.
Note that the standard requirement of existence of simplicial path functors is stronger
than what we require: we do not require representability of the functor
$$ Y\mapsto\Map(K,\Map_\cC(Y,X)).$$
We will call our requirement {\sl the existence of weak path functors}.

It is enough to require representability of the functors~(\ref{eq:weakpath}) for $K=\Delta^n$.
Then one will automatically have $X^K=\lim X^\Delta$ where $X^\Delta$ is the functor
from the category of simplices in $K$ to $\cC$ carrying $\Delta^n\to K$ to $X^{\Delta^n}$.

The functors $\Lambda^K:X\mapsto X^K$ for a fixed $K$ have automatically a 
structure of monad coming from the composition law in $\cC$. In fact, the composition map
\begin{equation}
\Map_\cC(Y,X)\times\Map_\cC(Z,Y)\rTo\Map_\cC(Z,X)
\end{equation}
yields a collection of maps 
\begin{equation}
\Hom_\cC(Y,X^K)\times\Hom_\cC(Z,Y^K)\rTo\Map_\cC(Z,X^K)
\end{equation}
which, applied to $Y=X^K$, yields, in particular, a canonical map
\begin{equation}
\Hom_\cC(Z,(X^K)^K)\rTo\Map_\cC(Z,X^K),
\end{equation}
that is, a canonical map $\Lambda^K\circ\Lambda^K\to\Lambda^K$.
The unit of the monad is defined by the canonical map $X\to X^K$.

Furthermore, the maps $\Lambda^K\to\Lambda^{K\times L}$ and $\Lambda^L\to\Lambda^{K\times L}$ 
yield 
\begin{equation}\label{eq:lambdaprod}
\Lambda^K\circ\Lambda^L\rTo\Lambda^{K\times L}\circ\Lambda^{K\times L}\rTo\Lambda^{K\times L}.
\end{equation}

The following lemma is obvious.
\begin{Lem}
A simplicial category with weak path functors admits simplicial path functors
(in the sense of Quillen) iff the maps (\ref{eq:lambdaprod}) are isomorphisms.
\end{Lem}
\qed

\begin{dfn}\label{dfn:weak}
Let $\cC$ be a model category having a simplicial structure. We call it 
{\sl a weak simplicial model category} if it admits weak path functors and satisfies the standard
(M7) condition of \cite{hirsch}, 9.1.6:
\begin{itemize}
\item[]
If $i:A\to B$ is a cofibration in $\cC$ and $p:X\to Y$ is a fibration in $\cC$, then the
map of simplicial sets
$$ \Map(B,X)\rTo\Map(A,X)\times_{\Map(A,Y)}\Map(B,Y)$$
is a fibration which is a trivial fibration if either $i$ or $p$ is a weak equivalence.
\end{itemize}
\end{dfn}

\begin{Exm}The category of complexes $C(A)$ over an associative ring $A$ has a projective model
structure (quasiisomorphisms as weak equivalences, componentwise surjective maps as fibrations).
It has also a simplicial category structure so that weak path functors exist:
the functor $Y\mapsto\Map(Y,X)_n$ is presented by the the complex $C^*(\Delta^n,\Z)\otimes_\Z X$,
where $C^*(\Delta^n,\Z)$ is the complex of normalized integral cochains on $\Delta^n$.
This is a weak simplicial model category.
\end{Exm}

\begin{Exm}(See, for instance, \cite{haha}, Sec.~4) Let now $k\supset\Q$ be a commutative ring and let $\cO$ be an operad in $C(k)$.
The category $\Alg_\cO(C(k))$ of $\cO$-algebras with values in $C(k)$
has a simplicial structure with weak path functors given by the formula
\begin{equation}
A^{\Delta^n}=\Omega_n\otimes A,
\end{equation}where $\Omega_\bullet$ is the simplicial algebra of polynomial differential forms
$$ n\mapsto \Omega_n=k[x_0,\ldots,x_n,dx_0,\ldots,dx_n]/(\sum x_i-1,\sum dx_i).$$
This is also a weak simplicial model category.
\end{Exm}

\subsubsection{}
\label{sss:prp}
In what follows we denote by $\cC_*=\{\cC_n\}$ and $\cC^{cf}_*$ the model category $\cC$ considered as a simplicial category and its full simplicial subcategory spanned by the 
fibrant-cofibrant objects.  

\begin{Prp}
Let $\cC_*$ be a weak simplicial model category. The following maps are weak equivalences of
simplicial categories.
\begin{itemize}
\item[0.] The localization map $\cC^{cf}_*\rTo L^H(\cC^{cf}_*)$. 
\item[1.] The maps $L^H(\cC_0)\rTo L^H(\cC_k)$ induced by the degeneracy
$\cC_0\to\cC_k$.
\item[2.] The maps $L^H(\cC^f_0)\rTo L^H(\cC^f_k)$.  
\item[3.] The maps $L^H(\cC^{cf}_k)\rTo L^H(\cC^f_k)$. 
\end{itemize}
 \end{Prp}
The proof will be given in \ref{sss:proof-prp}.

\begin{Crl}Let $\cC_*$ be a weak simplicial model category.
Then the maps of simplicial categories
\begin{equation}
\cC^{cf}_*\rTo L^H(\cC^{cf}_*)\rTo L^H(\cC^f_*)\rTo L^H(\cC_*)\lTo L^H(\cC)
\end{equation}
are equivalences.
\end{Crl}
\qed

\subsubsection{Proof of Proposition~\ref{sss:prp}} 
\label{sss:proof-prp}

\

0. This follows from the description of localization via the universal property.
Since $W^{cf}_*$ is a simplicial groupoid, the map $W^{cf}_*\rTo L^H(W^{cf}_*)$
is a weak equivalence, and this implies the claim.

1. Define a functor $\Lambda^k:\cC_k\rTo\cC$ as follows. For $X\in\cC_k$ let
$\Lambda^k(X)$ be $X^{\Delta^k}$. A map $f:X\to Y$ in $\cC_k$ is given by a map
$\phi:X\to Y^{\Delta^k}$. It yields a composition
$$ X^{\Delta^k}\rTo^{\phi^{\Delta^k}}(Y^{\Delta^k})^{\Delta^k}\rTo Y^{\Delta^k}$$
which will be $\Lambda^k(f)$. The functor $\Lambda^k$ so defined is right adjoint 
to the unit functor $U:\cC_0\to\cC_k$ carrying $X$ to $X$ and $f:X\to Y$ to
$X\rTo^fY\rTo Y^{\Delta^k}$. The unit and the counit of the adjunction being in $W$,
the adjunction induces an equivalence of DK localizations.

2. The pair $(U,\Lambda^k)$ defines also an adjunction of 
$\cC^f$ and $\cC^f_k$.

3. The proof uses the Key lemma very similarly to the proof of \ref{prp:DK5.2}. 

The category $\wt\cC$ consists of the weak equivalences $P\to X^{\Delta^k}$ where 
$X$ is fibrant and $P$ is fibrant cofibrant. Morphisms from $P\to X^{\Delta^k}$ to $Q\to X^{\Delta^k}$  are given by commutative triangles in $\cC_k$.

Let now $\sigma=(X_0\to\ldots\to X_n)$. We have to prove that the nerve
$N(\wt\cC_\sigma)$ is weakly contractible. We denote by $X$ be object of $\cC^{[n]}$
corresponding to $\sigma$. A special object in $\wt\cC_\sigma$ is just a trivial fibration $q:Q\to X^{\Delta^k}$ in $\cC^{[n]}$ with cofibrant $Q$.
If $p:P\to X^{\Delta^k}$ is any object in $\wt\cC_\sigma$, and $q:Q\to X^{\Delta^k}$
a special object, there exists a map $P\to Q$ in $\cC$ (and so in $\cC_k$) making the diagram commutative.
This proves the nerve of $\wt\cC_\sigma$ is connected.
We will now verify that any pair of maps to a special object has the same image in 
the total localization. Once more, given $P\to X^{\Delta^k}$ in $\wt\cC_\sigma$,
we construct a cylinder object (\ref{eq:cyl}).
Now, given two map $a_0,a_1:P\rTo Q^{\Delta^k}$ in $\wt\cC_\sigma$ with special $Q$,
we can extend it to a map $a:\wt P\to Q^{\Delta^k}$. This proves any two arrows 
to a special object in $\wt\cC_\sigma$ have the same image in the localization.
This implies that the nerve of $\wt\cC_\sigma$ is simply connected.
Vanishing of the reduced homology of $\wt\cC_\sigma$ is proven in the same way as
in \ref{prp:DK5.2}. 
 
Proposition is proven.

\subsection{Quillen pair}
\label{ss:nerv-adj}

Let $F:\cC\rlarrows\cD:G$ be a Quillen pair.
In case of simplicial model categories and simplicial adjunction,
this induces a pair of adjoint functor between the underlying $\infty$-categories, see \cite{L.T}, Proposition 5.2.4.6. Proposition~\ref{prp:nerve-adj} below asserts 
that one does not really need the simplicial structure here.

The functor $F$ preserves weak equivalences between cofibrant objects, and $G$
preserves weak equivalences between fibrant objects. This defines by 
universality a pair of functors which we denote for obvious reasons as the 
derived functors,
\begin{equation}
\Left F:N(\cC)\rlarrows N(\cD):\Right G.
\end{equation}

\begin{prp}\label{prp:nerve-adj}
The functors $\Left F$ and $\Right G$ form an adjoint pair of functors 
between $\infty$-categories.
\end{prp}

\begin{proof}
According to \cite{L.T}, 5.2.2, a pair of adjoint functors is defined by 
an $\infty$-category which is both cartesian and cocartesian fibration over 
$\Delta^1$.

Define a simplicial category $\cM$ over $\Delta^1$ as follows. The objects
of $\cM$ over $0$ are the cofibrant objects of $\cC$, and the objects over $1$
are the fibrant objects of $\cD$. We denote as $c,c',\ldots$ the objects
over $0$ and as $d,d',\ldots$ the objects over $1$.

In what follows we use the following notation. Let $C$ be a simplicial
category. Applying to all simplicial Hom-sets the functor
$$ X\mapsto \Sing|X|,$$
we get a functorial fibrant replacement $C^\phi$ of $C$.

We define $\Map_\cM(c,c')$ as $\Map_{L^H(\cC^c)^\phi}(c,c')$ and
$\Map_\cM(d,d')$ as $\Map_{L^H(\cD^f)^\phi}(d,d')$. Furthermore, we put
$\Map_\cM(d,c)=\emptyset$ and $\Map_\cM(c,d)=\Map_{L^H(\cD)^\phi}(F(c),d)$.

The composition is defined by the simplicial functors 
$$L^H(\cD^f)^\phi\rTo L^H(\cD)^\phi \textrm{ and } 
L^H(\cC^c)^\phi\rTo L^H(\cD)^\phi,$$
the first one being an equivalence and second one being induced by $F$.
The simplicial category $\cM$ defined above is obviously fibrant.

The fiber of $\cM$ at $0$ is $L^H(\cC^c)^\phi$ whereas the fiber at $1$ is
$L^H(\cD^f)^\phi$. 
 
It remains to check that the functor $\cM\to\Delta^1$ is a cartesian and a 
cocartesian fibration.

According to \cite{L.T}, 5.2.4.4, we have to find for each object $c$
over $0$ an arrow $\alpha:c\to d$ and for each $d$ over $1$ an
arrow $\beta:c\to d$, so that
\begin{itemize}
\item For any $c'$ over $0$ the map $\Map_\cM(c',c)\to\Map_\cM(c',d)$,
induced by $\beta$, is an equivalence.
\item For any $d'$ over $1$ the map $\Map_\cM(d,d')\to\Map_\cM(c,d')$,
induced by $\alpha$, is an equivalence.
\end{itemize} 
The arrow $\alpha:c\to d$ is defined by a fibrant replacement $F(c)\to d$
whereas the arrow $\beta:c\to d$ is defined  by a cofibrant
replacement $c\rTo G(d)$ which is chosen to be a trivial fibration
(so that $c$ is in particular fibrant).

\

Let us check the requirements.
The universality of $\alpha$ is immediate as a weak equivalence $F(c)\to d$
gives rise to an equivalence of the map spaces in $L(\cD)^\phi$.

Universality of $\beta$ is slightly less obvious. We have to deduce that
the canonical map
\begin{equation}
\label{eq:compos}
\Map_\cM(c',c)\rTo\Map(F(c'),F(c))\rTo \Map(F(c'),d)=\Map(c',d)
\end{equation}
is an equivalence. This is proven as follows.  Choose a cosimplicial resolution $P^\bullet\to c'$;
We decompose the map $F(c)\rTo d$ adjoint to the cofibrant replacement 
$c\to G(d)$, into a trivial cofibration
followed by a fibration as shown below.
\begin{equation}
F(c)\rTo d'\rTo d.
\end{equation}

We have a commutative diagram of simplicial sets
\begin{equation}\label{eq:cd}
\begin{diagram}
\Hom_\cC(P^\bullet,c) & \rTo & \Hom_\cD(F(P^\bullet),F(c))&\rTo&
\Hom_\cD(F(P^\bullet),d')\\
\dTo && && \dTo \\
\Hom_\cC(P^\bullet,G(d)) && \rEqual && \Hom_\cD(F(P^\bullet),d)
\end{diagram}
\end{equation}
which represents a commutative diagram
\begin{equation}\label{eq:adj}
\begin{diagram}
\Map_{L^H(\cC^c)^\phi}(c',c) & \rTo & \Map_{L^H(\cD)^\phi}(F(c'),F(c)) \\
\dTo & & \dTo \\
\Map_{L^H(\cC)^\phi}(c',G(d)) & \rEqual & \Map_{L^H(\cD)^\phi}(F(c'),d)
\end{diagram}.
\end{equation}
Since the left vertical map is obviously an 
equivalence, the composition (\ref{eq:compos}) is also an equivalence as 
required.

\end{proof}

\begin{crl}Let $\cC$ be a combinatorial model category. Then the underlying $\infty$-category 
$N\cC$ is presentable. The limits and colimits in $N\cC$  can be calculated as derived limits and colimits in $\cC$.
\footnote{Presentability of $N\cC$ is proven in \cite{L.HA}, 1.3.4.22.
Lurie defines $N\cC$ as the localization of $\cC^c$ which is of course 
equivalent to our definition.} 
\end{crl}
\begin{proof}
According to Dugger's theorem~\cite{D}, Corollary 1.2, any combinatorial
model category is Quillen equivalent to a simplicial combinatorial model category whose underlying $\infty$-category is known to be presentable.
Since Quillen equivalent model categories have equivalent underlying $\infty$-categories, this proves presentability of $N\cC$ in general.

Let now $I$ be a category. The category of functors $\cC^I$ has injective and projective model structures. One has two Quillen pairs,
\begin{equation}
\colim: \cC^I\rlarrows \cC: c
\end{equation}
and
\begin{equation}
c:\cC\rlarrows\cC^I:\lim,
\end{equation}
where the functor $c$ assigns to $x\in\cC$ the constant diagram with value $x$. In the first Quillen pair $\cC^I$ is endowed with the injective model structure, and in the second one with the projective model structure. 

It remains to show that the $\infty$-category underlying $\cC^I$ (in either model structure) is equivalent to $\Fun(N(I),N(\cC))$.
Here once more we use Dugger's result. Any Quillen equivalence
$\cC\lrarrows\cD$ of combinatorial model categories gives rise to a Quillen equivalence 
$\cC^I\lrarrows\cD^I$. Thus, having in mind Dugger's theorem, Corollary 1.2,
one can assume that $\cC$ is a combinatorial simplicial model category. In this case the claim
is a special case of \cite{L.T}, Proposition 4.2.4.4.
\end{proof}

\section{Localization in families}
\label{ss:quasi}

Since $\infty$-localization is functorial, it is reasonable to expect its nice behavior 
in families. In this section we assert that for a nice family of marked $\infty$-categories,
localization of the fibers is equivalent to fibers of the map of the localization
\footnote{A similar result for conventional categories was independently obtained by
Haugseng \cite{Hau}.}.

The following definition describes a notion of a (marked) family
of marked infinity categories.

\subsection{}
Recall that $\Cat^+_\infty$ is the $\infty$-category of marked 
$\infty$-categories (markings are assumed to be saturated).

\begin{dfn}\label{dfn:mccf}
An arrow $f:(C,V)\rTo(D,W)$ in $\sSet^+$ is called 
{\sl marked cocartesian fibration}
\footnote{Caution: our definition differs from \cite{L.GI}, 1.4.9}
if the following properties are fulfilled.
\begin{itemize}
\item[1.] $f:C\to D$ is a cocartesian fibration of $\infty$-categories.
\item[2.] A cocartesian lifting of a marked arrow in $D$ is marked in $C$.
\item[3.] For any arrow $\alpha:d\to d'$ in $D$ the functor 
$\alpha_!: C_d\to C_{d'}$ preserves marked arrows.
\item[4.] If $\alpha:d\to d'$ is marked then $\alpha_!$ induces an equivalence of
the localizations $\cL(C_d,V\cap C_d)\rTo\cL(C_{d'},V\cap C_{d'})$.
\end{itemize}
\end{dfn}

\begin{rem}
\label{rem:marked-other}
The markings $V\subset C$ are uniquely defined by their
intersection with $f^{-1}(K(D))$ as any marked arrow in $C$ decomposes
into a a cocartesian lifting of its image in $D$ and a marked arrows
whose image in $D$ is equivalence.
\end{rem}

Our main result Proposition \ref{prp:quasi} below describes a family of localizations of the fibers $(C_d,C_d\cap V)$. In order to formulate it, we need a more "homotopy invariant" version of the notion of cocartesian fibration. This is what 
Lurie~\cite{L.GI} calls ``essentially a cocartesian fibration'', and we prefer to call just {\sl a cocartesian fibrations in $\Cat_\infty$}. Here it is.

\begin{dfn}A map $f:C\to D$ in $\Cat_\infty$  is called 
{\sl a cocartesian fibration} if it is equivalent to a map represented by a
cocartesian fibration $f':C'\rTo D'$ in $\sSet$.
\end{dfn}

A morphism $f:C\rTo D$ of $\infty$-categories in $\sSet$ represents a cocartesian fibration
in $\Cat_\infty$ if and only if it
can be embedded into a homotopy commutative diagram
\begin{equation}
\begin{diagram}
C && \rTo^i && C' \\
& \rdTo^f& & \ldTo^g \\
&& D &&
\end{diagram},
\end{equation}
where $g$ is a cocartesian fibration and $i$ is a categorical equivalence. Moreover, if $f:C\to D$ is a categorical fibration
presenting a cocartesian fibration in $\Cat_\infty$, it is a 
cocartesian fibration in $\sSet$, see \cite{L.GI}, 1.4.5.
 
Let $f:C\rTo D$ be a cocartesian fibration in$\Cat_\infty$. An arrow
$\alpha:\Delta^1\to C$ is $f$-cocartesian if its composition with $i:C\to C'$ as in the above diagram, is $g$-cocartesian. This notion is independent of presentation and 
$f$-cocartesian arrows in $C$ form a subcategory.

\ 

We are now able to formulate the main result of this section. This result
is used in \cite{H.R}, Section 4.

\begin{prp}\label{prp:quasi}
Let $f:(C,V)\to (D,W)$ be a marked cocartesian fibration. Then the localization 
$$ \cL(f):\cL(C,V)\rTo\cL(D,W)$$
is a cocartesian fibration in $\Cat_\infty$. Moreover, for any $d\in D$ the induced map 
from $(C_d,V\cap C_d)$ to the homotopy fiber of $\cL(f)$ at $d$, 
is an $\infty$-localization.
\end{prp}

The proof is given in \ref{sss:beginprop}---\ref{sss:endprop} below.

\subsubsection{}
\label{sss:beginprop}

For an $\infty$-category $D$ we define  
$\Coc(D)$ as the subcategory of $(\Cat_\infty)_{/D}$ spanned by the cocartesian fibrations $C\to D$, with the maps preserving cocartesian arrows. 

Otherwise, $\Coc(D)$ can be described as the $\infty$-category underlying the
category $\sSet^+_{/D}$ endowed with the cocartesian model structure, see \cite{L.T}, Chapter 3.
 
Similarly, for $(D,W)\in\Cat^+_\infty$, the infinity category 
$\Coc^+(D,W)$ can be defined  
as the subcategory of $(\Cat^+_\infty)_{/(D,W)}$ spanned by the marked
cocartesian fibrations, with the maps preserving cocartesian arrows.
 It will be convenient, however, to identify $\Coc^+(D,W)$ with another subcategory of $(\Cat^+_\infty)_{/(D,W)}$, taking into account Remark~\ref{rem:marked-other}, via the functor 
\begin{equation}\label{eq:another-embedding}
r:\Coc^+(D,W)\rTo(\Cat^+_\infty)_{/(D,W)}
\end{equation}
carrying $f:(C,V)\to(D,W)$ to $r(f):(C,V\times_DK(D))\to(D,W)$.

\

We will use a weak form of straightening/unstraightening equivalence described in 
\cite{L.T}, Chapter 3. It provides for an $\infty$-category $D$ an equivalence of
$\infty$-categories
\begin{equation}\label{eq:str-unstr}
\Coc(D)\simeq \Fun(D,\Cat_\infty).
\end{equation}

The marked version of the above equivalence is described as follows.
First of all, for $(C,V), (D,W)\in\Cat^+_\infty$ we denote
$\Fun((C,V),(D,W))$ as the full subcategory of $\Fun(C,D)$ spanned by the functors
carrying $V$ to $W$.

Let $\Lambda$ be the collection of arrows in $\Cat^+_\infty$ carried by 
$\cL:\Cat^+_\infty\rTo\Cat_\infty$ to equivalence.

\begin{lem}\label{lem:str-unstr-marked}
The equivalence (\ref{eq:str-unstr}) induces an equivalence
\begin{equation}\label{eq:str-unstr-plus}
\Coc^+(D,W)\simeq \Fun((D,W),(\Cat^+_\infty,\Lambda)).
\end{equation}
\end{lem}
\begin{proof}
The right-hand side of~(\ref{eq:str-unstr-plus}) is a full 
subcategory of $\Fun(D,\Cat^+_\infty)$ which is a full subcategory 
of $\Fun(\Delta^1,\Fun(D,\Cat_\infty))=\Fun(\Delta^1,\Coc(D))$.  
The latter is  a subcategory of $\Fun(\Delta^1,(\Cat_\infty)_{/D})$.

We identify $\Coc^+(D,W)$ with a subcategory of $(\Cat^+_\infty)_{/(D,W)}$ using the functor $r$ described in~(\ref{eq:another-embedding}).
The latter is a subcategory of $\Fun(\Delta^1,(\Cat_\infty)_{/D})$. 

It remains to note that two sides of the 
formula~(\ref{eq:str-unstr-plus}) determine the same subcategory 
of $\Fun(\Delta^1,(\Cat_\infty)_{/D})$.
 
\end{proof}
 
\subsection{Proof of \ref{prp:quasi}}
 
We will use the following simple lemma.
\begin{lem}\label{lem:base-change}
Let $f:X\to S$ be a cocartesian fibration in $\Cat_\infty$,\\ $\alpha:I^\triangleright\to (\Cat_\infty)_{/S}$ be
a colimit diagram, and let $\beta:I^\triangleright\to(\Cat_\infty)_{/X}$ be obtained from $\alpha$ by a base change along $f$. Then $\beta$ is also a colimit diagram.
\end{lem}
\begin{proof}
We can represent $f$ with a cocartesian fibration of $\infty$-categories and $\alpha$
with a cofibrant representative in the projective model structure on $\Fun(I,\sSet)$,
where $\sSet$ is endowed with the Joyal model structure.
Then the naive colimit in $\sSet$ of $\alpha$ followed by a fibrant replacement, represents the colimit of $\alpha$ in $(\Cat_\infty)_{/S}$. The base change of a cofibrant object is cofibrant,
it commutes with naive colimits, and preserves weak equivalences, see \cite{L.T}, 3.3.1.3. This implies the claim. 
\end{proof}

\subsubsection{}
The marked cocartesian fibration $f:(C,V)\to (D,W)$ is classified by a functor   $F^+:D\to\Cat^+_\infty$ carrying $W$ to $\Lambda$.
The composition $F=\phi\circ F^+:D\rTo\Cat_\infty$ of $F^+$ with the forgetful functor
classifies the cocartesian fibration $C$ over $D$.

The composition $\cL\circ F^+:D\rTo\Cat_\infty$ carries $W\subset D$ to equivalences,
so it extends to a functor
\begin{equation}
\bF:\cL(D,W)\rTo\Cat_\infty,
\end{equation}
which can be converted back to a cocartesian fibration $\hat f:X\rTo\cL(D,W)$.

The canonical map of functors $F\rTo\bF|_D$ induced by the $\infty$-localization leads to a map
$C\to X$ over $D\to\cL(D,W)$ carrying $V$ to equivalences. This yields
a canonical map 
\begin{equation}\label{eq:theta}
\theta:\cL(C,V)\rTo X
\end{equation}
over $\cL(D,W)$. Our aim is to verify $\theta$ is 
an equivalence. In other words, we have to verify that
 for any $Y\in\Cat_\infty$  the map $\theta$ defines an equivalence
$$ \Map(X,Y)\rTo\Map^\sharp((C,W),Y^\natural).$$

\

Denote $X_D=D\times_{\cL(D,W)}X$. The map $X_D\rTo D$ is a cocartesian fibration
classified by the functor $\cL\circ F^+:D\rTo\Cat_\infty$. 

An arrow $\alpha\in V$ will be called {\sl vertical} if $f(\alpha)$ is identity.
It is called {\sl horizontal} if it is a cocartesian lifting of an arrow in $W\subset D$.
We denote $V^\hor$ and $V^\ver$ the collection of horizontal, resp., of vertical marked arrows.
By definition, the set $V\subset C$, in a marked cocartesian fibration, is generated by 
$V^\hor\cup V^\ver$. One has natural maps $\cL(C,V^\ver)\rTo X_D$ and $\cL(X_D,V^\hor)\rTo X$.
We will prove that both are weak equivalences.

\subsubsection{The map $\cL(X_D,V^\hor)\to X$.}

First of all, let us verify the claim in the special case $W=D_1$, that is, $(D,W)=D^\sharp$.
The localization $\cL(X_D,V^\hor)$ can be interpreted in this case as the colimit of the
functor $\cL\circ F^+:D\to\Cat_\infty$ classifying the cocartesian fibration $X_D\to D$.

Similarly, $X$ can be interpreted as the colimit of the functor $\bF:\cL(D,D_1)\to\Cat_\infty$. The map $D\rTo \cL(D,D_1)$ is cofinal \cite{L.T}, 4.1.1.1, so the natural map of colimits is an equivalence.

The case of general marking $W\subset D_1$ can now be easily deduced. Denote $\bW$ the subcategory of $D$ consisting of the simplices whose all edges are in $W$.
The localization $\cL(D,W)$ is a pushout of the diagram $\cL(\bW)\lTo \bW\rTo D$.
According to Lemma~\ref{lem:base-change},
the $\infty$-categories
\begin{equation}
\begin{diagram}
X_\bW & \rTo & X_{\cL(\bW)} \\
\dTo & & \dTo \\
X_D & \rTo & X
\end{diagram}
\end{equation}
defined as pullbacks of $\bW,\cL(\bW)$ and $D$, also form a pushout diagram with $X$. 
According to the special case verified above, $X_{\cL(\bW)}$ is a localization of $X_\bW$
with respect to $V^\hor$. This also implies that $X$ is a localization of $X_D$ with respect
to $V^\hor$.

\subsubsection{The map $\cL(C,V^\ver)\to X_D$.}

We will verify the claim for $D=\Delta^n$. This will imply by Lemma~\ref{lem:base-change}
the claim for a general $D$ as both $\cL(C,V^\ver)$ and $X_D$ are presented as colimits of their base changes with respect to
$\Delta^n\to D$,
\begin{equation}
\cL(C,V^\ver)=\colim_{\Delta^n\to D}\cL(C\times_D\Delta^n,V^\ver),
\end{equation}

\begin{equation}
X_D=\colim_{\Delta^n\to D}X_D\times_D\Delta^n.
\end{equation}

The cocartesian fibration $C\to D$ classified by a functor $\Delta^n\rTo\Cat_\infty$,
given by a sequence $\overline C: C^0\rTo\ldots\rTo C^n$ of $\infty$-categories,
is equivalent to its {\sl mapping simplex} $M(\overline C)$ whose $k$-simplices over 
$\alpha:\Delta^k\to\Delta^n$ are just the $k$-simplices of $C^{\alpha(0)}$, see~\cite{L.T}, 3.2.2.7.

Denote $V^i=V^\ver\cap C^i$. We will mark an edge in $M(\overline C)$ over an edge $i\to j$ of $\Delta^n$ if it comes from a marked edge in $C^i$. The corresponding marked simplicial set
will be denoted $M(\overline C)^\natural$. 
The cocartesian fibration $X_D\to D$ is equivalent to the mapping simplex $M(\cL(\overline C))$ where
$$
\cL(\overline C): \cL(C^0,V^0)\rTo\ldots\rTo\cL(C^n,V^n).
$$
We want to verify that the map $M(\overline C)\to M(\cL(\overline C))$ induces, for each
$\infty$-category $Y$, a homotopy equivalence between $\Map(M(\cL(\overline C)),Y)$ and 
$\Map^\sharp_{\sSet^+}(M(\overline C)^\natural,Y^\natural)$. 

This results from the following presentation of the mapping simplex as colimit.

\begin{lem}
Let $\overline C: C^0\rTo\ldots\rTo C^n$ be a sequence of $\infty$-categories.
Then the mapping simplex $M(\overline C)$ can be presented as the colimit of the diagram
\begin{equation}
C^0\times\Delta^n\lTo C^0\times\Delta^{n-1}\rTo C^1\times\Delta^{n-1}\lTo C^1\times\Delta^{n-2}
\rTo\ldots\rTo C^n,
\end{equation}
where the forward arrows are defined by the maps $C^i\to C^{i+1}$ and the backward arrows are
defined by the $0$-th face maps $\Delta^i\to\Delta^{i+1}$.
\end{lem}  
\begin{proof}
Induction in $n$.
\end{proof}
 
\subsubsection{}
\label{sss:endprop}
Thus, the map $\theta$ defined in (\ref{eq:theta}) is an equivalence of $\infty$-categories. Therefore, it induces an
equivalence of homotopy fibers over any $d\in D$. Thus, homotopy fiber of $\cL(f)$ at $d$
is equivalent to $X_d=\bF(d)=\cL(C_d,V\cap C_d)$. 

Proposition~\ref{prp:quasi} is proven.

\section{Localization of SM $\infty$-categories}
\label{sec:SM}

\subsection{Adjoint functors}

Let $\cC^\otimes,\cD^\otimes$ be SM $\infty$-categories  and let 
$F:\cC^\otimes\to \cD^\otimes$ be a
symmetric monoidal functor. Recall that this means, in particular, 
that $p:\cC^\otimes\to N\Fin_*$ and $q:\cD^\otimes\to N\Fin_*$ are
cocartesian fibrations and that the functor $F$ preserves 
cocartesian edges.

Denote, as usual, $\cC=\cC^\otimes_{\langle 1\rangle}$ and similarly for $\cD$. 

The following lemma is a special case of \cite{L.HA}, 7.3.2.7.

\begin{lem}\label{lem:SM-andnot}
 The functor $F$ admits a right adjoint if and only if its restriction 
$F|_{\cC}$ admits a right adjoint. In this case the right adjoint functor to $F$ is 
automatically a morphism of $\infty$-operads.
\end{lem}

\begin{proof} 
The functor $F$ admits a right adjoint iff for any $d\in\cD^\otimes$ 
the presheaf on $\cC^\otimes$ defined as
$c\mapsto\Map(F(c),d)\in\cS$,
is representable.

Assume first that $d\in\cD$. In this case the above functor is represented by 
$G(d)$,
where $G$ is adjoint to $F|_\cD$. 

In fact, for $c=\bigoplus_{i\in I} c_i$ in the standard notation, with 
$c_i\in\cC$, for any $\alpha:I_*\to\langle 1\rangle$ in $\Fin_*$ one has
$$\Map^\alpha_{\cC^\otimes}(c,G(d))\simeq
\Map_\cC(\alpha_!(c),G(d))\simeq
\Map_\cC(\alpha_!(F(c)),d)\simeq
\Map^\alpha_{\cD^\otimes}(F(c),d).
$$

For a general $d=\bigoplus_{i\in I}d_i$ the functor
$c\mapsto\Map(F(c),d)$ is represented by $\bigoplus_{i\in I}G(d_i)$.

If $\alpha:d\to d'$ is an inert edge in $\cD^\otimes$, the formula
above for $G$ implies that $G(\alpha)$ is as well inert. This means that 
$G$ is automatically a map of $\infty$-operads.
\end{proof}

\subsection{Strict SM localization}
Let $p:\cC^\otimes\rTo N\Fin_*$ be a symmetric monoidal $\infty$-category \cite{L.HA}
with underlying category $\cC$ and let $W$ be a collection of 
arrows in $\cC$. 

\begin{dfn}\label{dfn:sm}
A (strict) SM localization of the pair $(\cC^\otimes,W)$ is a SM category $\cD^\otimes$ together with a SM functor
$f:\cC^\otimes\rTo\cD^\otimes$ carrying all arrows from $W$ to equivalences and satisfying the following properties.
\begin{itemize}
\item Universality: for any SM $\infty$-category $\cE$ the map
$$\Fun^{\SM}(\cD,\cE)\rTo \Fun^{\SM}_W(\cC,\cE)$$
from the space of SM functors $\cD\to\cE$ to the space of SM functors $\cC\to\cE$ carrying $W$
to equivalences, is an equivalence.
\item The map $(\cC,W)\rTo \cD$ is a Dwyer-Kan localization.
\end{itemize}
\end{dfn}
The marking $W$ of $\cC=\cC^\otimes_{\langle 1\rangle}$ defines
a marking of each fiber $\cC^\otimes_{\langle n\rangle}$ so that 
the maps
$$ \cC^\otimes_{\langle n\rangle}\rTo \cC^n$$
are equivalences of marked $\infty$-categories. The subcategory 
spanned by all marked arrows will be denoted $W^\otimes$. 
The following result is a direct consequence of \ref{prp:quasi}.
\begin{prp}
\label{prp:smloc}
Let $\cC^\otimes$ be a SM $\infty$-category, $W$ be a marking in $\cC$,
and $W^\otimes$ its extension as described above.
Assume that for any active arrow $\alpha$ in $\Fin_*$
the functor  $\alpha_!$ preserves markings. Then SM localization of $\cC^\otimes$
exists and is equivalent to the canonical map $\cL(\cC^\otimes,W^\otimes)\rTo N\Fin_*$.
\end{prp}

\subsection{Right SM localization}
The strict SM localization as defined above seldom exists, as the collection of arrows $W$ is seldom closed under tensor product. This is why we present below a more practical notion.

\begin{dfn}\label{dfn:left-sm}
Right SM localization of the pair $(\cC^\otimes,W)$ is a SM category $\cD^\otimes$ together with a lax SM functor $f:\cC^\otimes\rTo\cD^\otimes$ carrying all arrows from $W$ to equivalences and satisfying the following properties.
\begin{itemize}
\item Universality: for any SM $\infty$-category $\cE$ the map
$$\Fun^{\lax}(\cD,\cE)\rTo \Fun^{\lax}_W(\cC,\cE)$$
from the space of lax SM functors $\cD\to\cE$ to the space of lax SM functors $\cC\to\cE$ carrying $W$ to equivalences, is an equivalence.
\item The map $(\cC,W)\rTo \cD$ is a Dwyer-Kan localization.
\end{itemize}
\end{dfn}

Note that under the assumptions of Proposition~\ref{prp:smloc}
the SM localization $\cL(\cC^\otimes,W^\otimes)$ satisfies as well the universality with respect to lax SM functors, that is it is also a right SM localization. 

The following proposition describes another context where right SM localization exists. 

\begin{prp}\label{prp:right-sm-loc}
Let $\cC^\otimes$ be a SM $\infty$-category, $W$ be a marking in $\cC$,
and $W^\otimes$ its extension as described above. Assume that there exists a full SM subcategory 
$\cC^\otimes_0$ of $\cC^\otimes$ satisfying the following properties.
\begin{itemize}
\item
 The embedding $\iota:\cC_0\to\cC$ admits right adjoint $\rho:\cC\to\cC_0$ (that is, $\cC_0$ is a right (Bousfield) localization in terms of Lurie, see \cite{L.T}, 5.2.7.2).
\item  For any active arrow $\alpha$ in $\Fin_*$
the restriction of the functor  $\alpha_!$ to $\cC^\otimes_0$ preserves markings.
\item Any arrow $\phi$ in $\cC$ such that $\rho(\phi)$ is an equivalence, is in $W$.
\end{itemize} 
 Then right SM localization of $\cC^\otimes$ with respect to $W$ exists
and is equivalent to the canonical map $\cL(\cC^\otimes_0,W^\otimes\cap\cC^\otimes_0)\rTo N\Fin_*$.
\end{prp}
\begin{proof}
Since the left adjoint functor $\iota:\cC_0\rTo\cC$ is a restriction of a SM functor, its right
adjoint $\rho$ extends canonically to a lax SM functor $\rho:\cC^\otimes\rTo\cC^\otimes_0$.

According to \ref{prp:smloc}, $\cL(\cC^\otimes_0,W^\otimes\cap\cC^\otimes_0)$
is a SM localization of $\cC^\otimes_0$, so one has an equivalence
\begin{equation}
\Fun^\lax(\cL(\cC^\otimes_0,W^\otimes\cap\cC^\otimes_0),\cD)\rTo\Fun^\lax_{W\cap\cC_0}(\cC_0,\cD).
\end{equation}
It remains therefore to check that the lax SM functor $\rho$ induces an equivalence
\begin{equation}
\Fun^\lax_{W\cap\cC_0}(\cC_0,\cD)\rTo\Fun^\lax_W(\cC,\cD).
\end{equation}

The embedding $\iota:\cC_0\to\cC$ yields a map in the opposite direction.
Since the unit and the counit of the adjunction $\iota:\cC_0\rlarrows\cC:\rho$
belong to $W$, they prove the constructed maps are homotopy inverse to each other.
\end{proof}

Let $\cC$ be a symmetric monoidal category endowed with a structure of model category. 
In case the left derived tensor product defines a symmetric monoidal structure on $\Ho(\cC)$,
one has a lax SM functor $Q:\cC\rTo\Ho(\cC)$, so
we would like to expect that the underlying $\infty$-category $N(\cC)$ is a right SM localization
of $\cC$. We present below the only case we were able to prove.
 
\begin{exm}
\label{exm:complexes}
The category of complexes $C(k)$ over a commutative ring $k$ is symmetric monoidal.
Let us show that there exists a right SM localization of $C(k)$ with respect to quasiisomorphisms. Denote $C_*(k)$ the simplicial category of complexes of $k$-modules,
with the simplicial map space $\Map(X,Y)$ defined as in \ref{dfn:weak}.
 
The category $C_*(k)$ is a fibrant simplicial SM category and the embedding yields
a SM functor $\iota:C(k)\rTo C_*(k)$. We will denote by the same letter the
SM functor between the corresponding SM $\infty$-categories.

The full subcategory $C^c_*(k)$ of $C_*(k)$
spanned by the cofibrant complexes is a SM $\infty$-category and the embedding admits right adjoint. This easily follows from the fact that for a cofibrant replacement 
$\wt X\rTo X$ of a complex $X$ and any cofibrant complex $Y$ the natural map induced by the composition
\begin{equation}
\Map(Y,\wt X)\rTo\Map(Y,X),
\end{equation}
is an equivalence. This yields the lax SM functor $C(k)\rTo \cL(C^c_*(k))$.
It is universal by the same reasoning we used in the proof of Proposition~\ref{prp:right-sm-loc}.
\end{exm}


\begin{thebibliography}{MMMM}
\bibitem[BK1]{BK1} C.~Barwick, D.~Kan, Relative categories: another model for the homotopy theory of homotopy theories,
 Indag. Math. (N.S.) 23 (2012), no. 1-2, 42–68. 
\bibitem[BK2]{BK2} C.~Barwick, D.~Kan, 
\bibitem[BK3]{BK3} C.~Barwick, D.~Kan, 
\bibitem[Dug]{D} D.~Dugger, Combinatorial model categories have presentations,
Adv. Math. 164(2001), 177-201, arXiv.math/0007068.
\bibitem[DK1]{DK1} W.~Dwyer, D.~Kan, Simplicial localizations of categories. J. Pure Appl. Algebra 17 (1980), no. 3, 
267-284. 
\bibitem[DK2]{DK2} W.~Dwyer, D.~Kan, Calculating simplicial localizations. J. Pure Appl. Algebra 18 (1980), no. 1,
17-35. 
\bibitem[DK3]{DK3} W.~Dwyer, D.~Kan, Function complexes in homotopical algebra. Topology 19 (1980), no. 4, 
427-440. 
\bibitem[Hau]{Hau} R.~Haugseng, Rectification of enriched infinity-categories,
arXiv:1312.3881.
\bibitem[H.H]{haha}V.~Hinich, Homological algebra of homotopy aplebras,
Comm. Alg., 25:10 (1997), 3291--3323.
\bibitem[H.R]{H.R}V.~Hinich, Rectification of algebras and modules, preprint arXiv:1311.4130.
\bibitem[H.DSA]{DSA} V. Hinich, Deformations of sheaves of algebras, Adv. Math. 195(2005),
no. 1, 102--164, arXiv:math/0310116.
\bibitem[Hir]{hirsch} P.~Hirschhorn, Model categories and their localizations,  Mathematical Surveys and Monographs, 99. AMS, Providence, RI, 2003. xvi+457 pp.
\bibitem[JT]{JT} A.~Joyal, M.~Tierney, Quasi-categories versus Segal spaces. Categories in algebra, geometry and mathematical physics, 
277–-326, Contemp. Math., 431, Amer. Math. Soc., Providence, RI, 2007.
\bibitem[L.HA]{L.HA} J.~Lurie, Higher algebra, preprint August 3, 2012, available at http://www.math.harvard.edu/~lurie/papers/HigherAlgebra.pdf.
\bibitem[L.T]{L.T} J.~Lurie, Higher topos theory, Annals of Mathematics Studies, 170. Princeton University Press, Princeton, NJ, 2009. xviii+925 pp, also available at
http://www.math.harvard.edu/~lurie/papers/croppedtopoi.pdf.
\bibitem[L.G]{L.GI} J. Lurie, $(\infty,2)$-categories and the Goodwillie calculus, I, preprint, October, 2009, available at
http://www.math.harvard.edu/~lurie/papers/GoodwillieI.pdf.
\bibitem[R]{R} C.~Rezk, A model for the homotopy theory of homotopy theories, Transactions AMS, 353 (2001), 973--1007.
\bibitem[SP]{over}Mathoverflow discussion,
\url{http://mathoverflow.net/questions/92916/does-the-classification-diagram-localize-a-category-with-weak-equivalences/92941#92941}
\bibitem[T]{T} B.~Toen, Vers une axiomatisation de la th\'eorie des cat\'egories sup\'erieures,  K-Theory 34 (2005), no. 3, 233–263.

\end{thebibliography}
\end{document}